\documentclass[12pt,a4paper]{article}
\usepackage[latin1]{inputenc}
\usepackage{csquotes}
\usepackage{amsmath}
\usepackage{amsfonts}
\usepackage{amssymb}
\usepackage{amsthm}
\usepackage{psfrag}
\usepackage{graphicx}
\usepackage{enumitem}
\usepackage{indentfirst}
\usepackage{bbm}
\usepackage{xcolor}
\newcommand{\lel}{\left\langle}
\newcommand{\rir}{\right\rangle}
\newcommand{\diag}{\text{diag}}
\newcommand{\keywords}{{\it ~ Keywords:~}}
\newtheorem{theorem}{Theorem}[section]
\newtheorem{definition}[theorem]{Definition}
\newtheorem{lemma}[theorem]{Lemma}
\newtheorem{ass}[theorem]{Assumption}
\newtheorem{proposition}[theorem]{Proposition}
\makeatother
\title{Some Properties of Reflected Backward Stochastic Differential Equations for a Finite State Markov Chain Model}
\author{Zhe Yang \thanks{Department of Mathematics and Statistics, University of Calgary, 2500 University Drive NW, Calgary, AB, T2N 1N4, Canada.} \and Dimbinirina Ramarimbahoaka\footnotemark[1] \and Robert J. Elliott \thanks{Haskayne School of Business, University of Calgary, 2500 University Drive NW, Calgary, AB, T2N 1N4, Canada.}  \thanks{School of Mathematical Sciences, University of Adelaide, SA 5005, Australia.}}
\date{}
\begin{document}
\maketitle
\begin{abstract}
In this paper, we
provide an estimate for the solutions of reflected backward stochastic differential equations (RBSDEs) driven by a Markov chain, derive a continuous dependence property for their solutions with
respect to the parameters of the equations, and show similar properties for solutions of backward stochastic differential equations (BSDEs). We finally establish a comparison result
for the solutions of RBSDEs driven by a Markov chain.
\end{abstract}


\section{Introduction}	
 \indent In 2012, van der Hoek and Elliott \cite{RE1} introduced a
market model where uncertainties are modeled by a finite state
Markov chain, instead of Brownian motion or related jump diffusions,
which are often used when pricing financial derivatives. The
Markov chain has a semimartingale representation involving a vector martingale $M=\{M_t\in\mathbb{R}^N,~t\geq 0\}$. BSDEs in this
framework were introduced by Cohen and Elliott \cite{Sam1} as
$$ Y_t = \xi + \int_t^T F(s, Y_s, Z_s) ds -\int_t^T  Z'_sdM_s
,~~~~~t\in[0,T],
$$
where $F$ is the driver, $\xi$ is the terminal condition and $M$ is a vector martingale given by the dynamics of the Markov chain.\\
\indent Cohen and Elliott \cite{Sam2} and \cite{Sam3} gave some comparison results for
multidimensional BSDEs in the Markov Chain model
under conditions involving not only the two drivers but also
the two solutions. Cohen and Elliott \cite{Sam3} also
showed the existence of solutions of the above equations with stopping times and introduced a type of nonlinear expectations called $F$-expectations, corresponding to the solution of these equations and  based on the comparison results. Yang, Ramarimbahoaka and Elliott \cite{zhedim} extended the comparison result for two one-dimensional BSDEs driven
by a Markov chain to a situation involving
conditions only on the two drivers and provided a converse comparison result in terms of $F$-expectations defined in  \cite{zhedim}. \\
\indent An, Cohen and Ji \cite{An} discuss American options using the theory of reflected backward stochastic differential equations (RBSDEs) with Markov chain noise in discrete time. Based on the comparison theorem in \cite{zhedim} and using the penalization method Ramarimbahoaka, Yang and Elliott \cite{RYE} establish the existence and uniqueness of the solution $(V,Z,K)$ of the following RBSDE:
\begin{enumerate}[label=\roman{*}), ref=(\roman{*})]
\item  $V_t = \xi + \int_t^T f(s, V_s, Z_s) ds + K_T -K_t -\int_t^T  Z'_s dM_s$, $~~~0 \leq t \leq T$;
\item  $V_t \geq G_t $,  $~~~0 \leq t \leq T$;
\item  $\{K_t, t \in [0,T]\}$ is continuous and increasing, moreover, $K_0=0$ and \\
$\int_0^T (V_s - G_s) dK_s= 0.$
\end{enumerate}
This is proven under some conditions on the terminal condition $\xi$, the driver $f$ and $G$, which is called an obstacle, and is a process to force the solution $V$ to stay above $G$. \\
\indent In this paper, we derive some properties of RBSDEs for the Markov chain noise case. We
provide an estimate for the solutions of this type of equation which establishes their boundness in some sence. We then discuss the difference between two solutions depending on the parameters of two equations, and deduce similar properties for solutions of BSDEs with Markov chain noise. We finally show comparison results
for the solutions of RBSDEs driven by a Markov chain martingale.\\
\indent The sections of the paper are as follows: In Section 2, we present the Markov chain model and some preliminary results. Section 3 establishes an estimate of the solutions of RBSDEs for the Markov Chain, and Section 4 discusses the continuous dependence
property of solutions of RBSDEs for the Markov chain. In the final section, we deduce a comparison result for one-dimensional RBSDEs driven by the Markov chain.
\section{The  Model and Some Preliminary Results}\label{prelim}
\subsection{The Markov Chain }
\indent Consider a finite state Markov chain. Following
\cite{RE1} and \cite{RE2} of van der Hoek and Elliott, we assume the
finite state Markov chain $X=\{X_t, t\geq 0 \}$ is defined on the
probability space $(\Omega,\mathcal{F},P)$ and the state space of
$X$ is identified with the set of unit column vectors $\{e_1,e_2\cdots,e_N\}$ in
$\mathbb{R}^N$, where $e_i=(0,\cdots,1\cdots,0) ' $ with 1 in the
$i$-th position.Take $\mathcal{F}_t=\sigma\{X_s ; 0\leq s \leq t\}$ to be
the $\sigma$-algebra generated by the Markov process $X=\{X_t\}$
and $\{\mathcal{F}_t\}$ to be its completed natural filtration. Since $X$ is a right continuous with left
limits (written RCLL) jump-process, then the filtration $\{\mathcal{F}_t\}$ is also
right-continuous. The  Markov chain has the semimartingale
representation:
\begin{equation}\label{semimartingale}
X_t=X_0+\int_{0}^{t}A_sX_sds+M_t.
\end{equation}
Here, $A=\{A_t, t\geq 0 \}$ is the rate matrix of the chain $X$ and
$M$ is a vector martingale (See Elliott, Aggoun and Moore
\cite{RE4}).
We assume the elements $A_{ij}(t)$ of $A=\{A_t, t\geq 0 \}$ are bounded. Then the martingale $M$ is square integrable.\\
\indent
For our Markov chain $X_t \in \{e_1,\cdots,e_N\}$, note that $X_t X'_t $ is the matrix $\diag(X_t)$. Also, from \eqref{semimartingale}
$dX_t = A_t X_t dt + dM_t$. Then, by the product rule for semimartingales, we obtain for any $t\in[0,T],$
\begin{align}\label{1}
\nonumber X_tX'_t &= X_0X'_0 + \int_0^t X_{s-} dX'_s + \int_0^t (dX_{s}) X'_{s-} + \sum_{0 < s \leq t} \Delta X_s \Delta X'_s \\
\nonumber  &= \diag(X_0) + \int_0^t X_s (A_sX_s)' ds + \int_0^t X_{s-} dM'_s \\
\nonumber& \quad + \int_0^t A_s X_s X'_{s-} ds + \int_0^t (dM_s) X'_{s-} + [X,X]_t\\
\nonumber
& = \diag (X_0) + \int_0^t X_s X'_s A'_sds + \int_0^t X_{s-} dM'_s \\
& \quad + \int_0^t A_sX_s X'_{s-} ds + \int_0^t (dM_s) X'_{s-} + [X,X]_t - \lel X,X\rir_t + \lel X,X \rir_t.
\end{align}
Recall, $\lel X, X\rir$ is the unique predictable process such that $[X,X]-\lel X,X \rir$ is a martingale and write $L$ for the matrix martingale process
$$
L_t = [X,X]_t - \lel X,X\rir_t, \quad t \in [0,T].
$$
 However, we also have:
\begin{equation}\label{2}
X_tX'_t = \diag (X_t) = \diag(X_0) + \int_0^t \diag (A_s X_s)  ds + \int_0^t \diag(M_s).
\end{equation}
Equating the predictable terms in \eqref{1} and \eqref{2}, we have
\begin{equation}\label{3}
\lel X, X\rir_t =  \int_0^t \diag(A_sX_s) ds - \int_0^t \diag(X_s) A'_s ds - \int_0^t A_s \diag(X_s) ds.
\end{equation}
Let $\Psi$ be the matrix
\begin{equation}\label{Psi}\Psi_t = \diag(A_tX_{t-})- \diag(X_{t-})A'_t - A_t \diag(X_{t-}).
\end{equation}
Then $d\langle X,X\rangle_t=\Psi_tdt.$ For any $t>0$, Cohen and Elliott \cite{Sam1,Sam3}, define the semi-norm $\|.\|_{X_t}$, for
$C, D \in \mathbb{R}^{N\times K}$ as:
\begin{align*}
\lel C, D\rir_{X_t} & = Tr(C' \Psi_tD), \\[2mm]
\|C\|^2_{X_t} & = \lel C, C\rir_{X_t}.
\end{align*}
We only consider the case where $C \in \mathbb{R}^N$, hence we
introduce the semi-norm $\|.\|_{X_t}$ as:
\begin{align*}
\lel C, D\rir_{X_t}& = C' \Psi_t D, \\[2mm]
\|C\|^2_{X_t} &= \lel C, C\rir_{X_t}.
\end{align*}
It follows from equation \eqref{3} that
\[\int_t^T \|C\|^2_{X_s} ds = \int_t^T  C' d\lel X, X\rir_s C.\]
\indent Lemma \ref{Z2} is Lemma 3.1 in Cohen and Elliott \cite{Sam3}.
\begin{lemma}\label{Z2}
For $Z$, a predictable process in $\mathbb{R}^N$, verifying:
 \[E \left[ \int_0^t \|Z_u\|^2_{X_u} du\right] < \infty,\]
we have:
\begin{equation*}
 E \left[\left(\int_0^t  Z'_{u} dM_u  \right)^2\right] = E \left[ \int_0^t \|Z_u\|^2_{X_u} du\right].
\end{equation*}
\end{lemma}
\begin{definition}[Moore-Penrose pseudoinverse]\label{defMoore}
The Moore-Penrose pseudoinverse of a square matrix $Q$ is the matrix $Q^{\dagger}$ satisfying the properties:\\[2mm]
  1) $QQ^{\dagger}Q = Q$ \\[2mm]
  2) $Q^{\dagger}QQ^{\dagger} = Q^{\dagger}$ \\[2mm]
  3) $(QQ^{\dagger})' = QQ^{\dagger}$ \\[2mm]
  4) $(Q^{\dagger}Q)'=Q^{\dagger}Q.$
\end{definition}
Denote by $\mathcal{P}$, the $\sigma$-field generated by the predictable processes defined on $(\Omega, P, \mathcal{F})$ and with respect to the filtration $\{\mathcal{F}_t\}_{t \in [0,\infty)}$. For $t\in[0,\infty)$, consider the following spaces:\\[2mm]
$ L^2(\mathcal{F}_t): =\{\xi;~\xi$ is a $ \mathbb{R} \text{-valued}~ \mathcal{F}_t $-measurable random variable such that $ E[|\xi|^2]< \infty\};$\\[2mm]
$L^2_{\mathcal{F}}(0,t;\mathbb{R}): =\{\phi:[0,t]\times\Omega\rightarrow\mathbb{R};~ \phi$ is an adapted and RCLL process with  $E[\int^t_0|\phi(s)|^2ds]<+\infty\}$;\\[2mm]
$P^2_{\mathcal{F}}(0,t;\mathbb{R}^N): =\{\phi:[0,t]\times\Omega\rightarrow\mathbb{R}^N;~ \phi $ is a predictable process with  $E[\int^t_0\|\phi(s)\|_{X_s}^2ds]<+\infty\}.$\\[2mm]
\subsection{BSDEs for the Markov Chain Model.}\label{bsdeMC}
\indent Consider a one-dimensional BSDE with the Markov chain noise
as follows:
\begin{equation}\label{BSDEMC}
Y_t = \xi + \int_t^T f(u, Y_u, Z_u ) du -\int_t^T  Z'_{u} dM_u
,~~~~~t\in[0,T].
\end{equation}
Here the terminal condition $\xi$ and the coefficient $f$ are known. \\
\indent Lemma \ref{existence} (Theorem 6.2 in Cohen and Elliott \cite{Sam1}) gives the existence and uniqueness result of solutions to the BSDEs
driven by Markov chains.
\begin{lemma}\label{existence}
Assume $\xi\in L^2(\mathcal{F}_T)$ and the predictable
function $f: \Omega \times [0, T] \times \mathbb{R} \times
\mathbb{R}^N \rightarrow \mathbb{R}$ satisfies a Lipschitz
condition, in the sense that there exists two constants $l_1, l_2>0$  such
that for each $y_1,y_2 \in \mathbb{R}$ and $z_1,z_2 \in
\mathbb{R}^{N}$, $t\in[0,T]$,
\begin{equation}\label{Lipchl}
|f(t,y_1,z_1) - f(t, y_2, z_2)| \leq l_1 |y_1-y_2| + l_2 \|z_1
-z_2\|_{X_t}.
\end{equation}
We also assume $f$ satisfies
$
 E [ \int_0^T |f^2(t,0,0)| dt] <\infty.$\\
 Then there exists a solution $(Y, Z)\in L^2_{\mathcal{F}}(0,T;\mathbb{R})\times P^2_{\mathcal{F}}(0,T;\mathbb{R}^N)$
to BSDE (\ref{BSDEMC}). Moreover, this solution is
unique up to indistinguishability for $Y$ and equality $d\langle
X,X\rangle_t$ $\times\mathbb{P}$-a.s. for $Z$.
\end{lemma}
\begin{ass}\label{ass0}
Assume the Lipschitz constant $l_2$ of the driver $f$ given in \eqref{Lipchl} satisfies  $$~~~~~~l_2\|\Psi_t^{\dagger}\|_{N \times N} \sqrt{6m}< 1, ~~~\text{ for any }~t \in [0,T],$$where $\Psi$ is given in \eqref{Psi} and $m>0$ is the bound of $\|A_t\|_{N\times N}$, for any $t\in[0,T]$.
\end{ass}
\indent The following lemma, which is a comparison result for BSDEs driven by a Markov chain, is found in Yang, Ramarimbahoaka and Elliott \cite{zhedim}.\\
\begin{lemma} \label{CTBSDE} For $i=1,2,$ suppose $(Y^{(i)},Z^{(i)})$ is the solution of
BSDE:
$$Y^{(i)}_t = \xi_i + \int_t^T f_i(s, Y^{(i)}_s, Z^{(i)}_s ) ds
- \int_t^T (Z_{s}^{(i)})' dM_s,\hskip.4cmt\in[0,T].$$
Assume $\xi_1,\xi_2\in L^2(\mathcal{F}_T)$, and $f_1,f_2:\Omega \times [0,T]\times \mathbb{R}\times \mathbb{R}^N \rightarrow \mathbb{R}$ satisfy some conditions such that the above two BSDEs have unique solutions. Moreover assume $f_1$ satisfies \eqref{Lipchl} and Assumption \ref{ass0}.
If $\xi_1 \leq \xi_2 $, a.s. and $f_1(t,Y_t^{(2)}, Z_t^{(2)}) \leq f_2(t,Y_t^{(2)}, Z_t^{(2)})$, a.e., a.s., then
$$P( Y_t^{(1)}\leq Y_t^{(2)},~~\text{ for any } t \in [0,T])=1.$$
\end{lemma}
\subsection{RBSDEs with the Markov Chain Noise}
\indent Ramarimbahoaka, Yang and Elliott \cite{RYE} introduced a reflected BSDE (RBSDE) for the Markov Chain and derived the existence and uniqueness of the solutions. This is an equation of the form:
\begin{enumerate}[label=\roman{*}), ref=(\roman{*})]
\item  $V_t = \xi + \int_t^T f(s, V_s, Z_s) ds + K_T -K_t -\int_t^T  Z'_s dM_s$, $~~~0 \leq t \leq T$;
\item  $V_t \geq G_t $,  $~~~0 \leq t \leq T$;
\item  $\{K_t, t \in [0,T]\}$ is continuous and increasing, moreover, $K_0=0$ and \\
$\int_0^T (V_s - G_s) dK_s= 0$.
\end{enumerate}
\begin{lemma} Suppose we have:
\begin{enumerate}
 \item $\xi\in L^2(\mathcal{F}_T)$,
\item  a $\mathcal{P} \times \mathcal{B}(\mathbb{R}^{1+N})$ measurable function $f: \Omega \times [0, T] \times \mathbb{R} \times \mathbb{R}^N \rightarrow \mathbb{R}$ which is Lipschitz continuous, with constants $c'$ and $c''$, in the sense that, for any $t \in [0,T]$, $v_1,v_2 \in \mathbb{R}$ and $z_1,z_2 \in \mathbb{R}^N$, $t\in[0,T]$,\emph{}
\begin{equation}\label{Lipch}
|f(t,v_1,z_1) - f(t, v_2, z_2)| \leq c'|v_1-v_2| + c''\|z_1 -z_2\|_{X_t}
\end{equation}
and $c''$ satisfies
 \begin{equation}\label{c''}c''\|\Psi_t^{\dagger}\|_{N \times N} \sqrt{6m}<1, ~~~\text{ for any }~t \in [0,T],\end{equation}
 where $\Psi$ is given in \eqref{Psi} and $m>0$ is the bound of $\|A_t\|_{N\times N}$, for any $t\in[0,T]$.
\item \begin{equation}\label{Con_f}
 E \left[ \int_0^T |f^2(t,0,0)| dt\right] < \infty,
\end{equation}
\item a process $G$ called an \enquote{obstacle} which satisfies
\begin{equation}\label{Con_g}
 E \left[ \sup_{0\leq t \leq T} (G_t^+)^2 \right] < \infty.
\end{equation}
\end{enumerate}
Then there exists a solution $(V,Z,K)$, $V$ adapted and RCLL and $Z$ predictable, of the RBSDE i), ii), iii) above such that $V \in L^2_{\mathcal{F}}(0,T;\mathbb{R})$, $K_T\in L^2(\mathcal{F}_T)$ and $Z \in P^2_{\mathcal{F}}(0,T;\mathbb{R}^N)$, moreover, this solution is
unique up to indistinguishability for $Y$, $K$ and equality $d\langle
X,X\rangle_t$ $\times\mathbb{P}$-a.s. for $Z$.
\end{lemma}

\section{Estimate of the solutions to RBSDEs for the Markov Chain}
\begin{proposition}\label{es}
Suppose $\xi\in L^2(\mathcal{F}_T)$, $f$ satisfies \eqref{Lipch}, \eqref{c''}, \eqref{Con_f} and $G$ satisfies \eqref{Con_g}.  Let $(V, Z,K)$ be the solution of the RBSDE for the Markov Chain satisfying:
$$ \left \{
\begin{array}{ll} V_t = \xi + \int_t^T f(s, V_s, Z_s ) ds + K_T -K_t -\int_t^T  Z_s' dM_s,~~  0 \leq t \leq T;\\[2mm]
 V_t \geq G_t ,~~~  0 \leq t \leq T;\\[2mm]
 \{K_t, t \in [0,T]\}\mbox{ is continuous and increasing, moreover, }K_0=0\mbox{ and} \\[1mm]
\int_0^T (V_s - G_s) dK_s = 0.\end{array} \right. $$
Then there exists a constant $C>0$ depending on the Lipschitz constants $c',c''$ of $f$ and $T$ such that
\begin{align*}
& E [ \sup_{0\leq t \leq T}|V_t|^2 ]+ E [ \int_0^T \|Z_t\|^2_{X_t} dt]+ E [|K_T|^2]\\
 & \leq C( E [|\xi|^2+\sup\limits_{0 \leq s \leq T} (G_s^+)^2+(\int_0^T |f(s, 0, 0)| ds)^2].
\end{align*}
\end{proposition}
\begin{proof} Applying the product rule to $|V_t|^2$, we derive for any $t\in[0,T],$
\begin{align}\label{v1v221}
\nonumber
|V_t|^2
& =|\xi|^2 -2 \int_t^T V_{s-} dV_{s}  - \sum_{t \leq s \leq T} \Delta V_s\Delta V_s \\
\nonumber
&= |\xi|^2+2 \int_t^T V_{s}  f(s,V_s, Z_s) ds+2 \int_t^T V_sdK_s \\
& \quad -2 \int_t^T V_{s-} Z_s' dM_s - \sum_{t \leq s \leq T} \Delta V_s\Delta V_s,
\end{align}
where
\begin{align}\label{deltav1v21}
\nonumber
\sum_{t\leq s \leq T}  \Delta V_s\Delta V_s
&= \sum_{t\leq s \leq T}(Z_s' \Delta X_s)( Z_s'\Delta X_s ) = \sum_{t\leq s \leq T} Z_s' \Delta X_s \Delta X_s' Z_s  \\
& =  \int_t^T Z_s'  (dL_s + d\lel X,X\rir_s) Z_s =  \int_t^T Z_s' dL_sZ_s + \int_t^T \|Z_s \|_{X_s}^2 ds.
\end{align}
As $\{K_t, t \in [0,T]\}$ is an increasing process, and moreover, for any $t \in[0,T]$, $V_t \geq G_t $, we know for any $t\in[0,T]$, $$0\leq\int_t^T (V_s- G_s) dK_s\leq\int_0^T (V_s- G_s) dK_s=0.$$ So we have
\begin{align}\label{k1k21}
  \int_t^T V_s dK_s
= \int_t^T  G_sdK_s,\quad t\in[0,T] .
\end{align}
By (\ref{v1v221}), (\ref{deltav1v21}) and (\ref{k1k21}), we obtain for any $t\in[0,T],$
$$\begin{array}{ll}
|V_t|^2
& = |\xi|^2+2 \int_t^T V_{s} f(s,V_s, Z_s) ds+2 \int_t^T  G_sdK_s \\[3mm]
& \quad-2 \int_t^T  V_{s-} Z_s' dM_s -  \int_t^T Z_s' dL_sZ_s- \int_t^T \|Z_s \|_{X_s}^2 ds.
\end{array}
$$
Now, let $\beta >0$ be an arbitrary constant and write $c=\max\{c',c''\}$. Using It$\hat{\text{o}}$'s formula for $e^{\beta t}|V_t|^2$, we deduce
 for any $t\in[0,T],$
$$\begin{array}{ll}
E [ e^{\beta t} |V_t|^2 ]+ E [ \int_t^T \beta |V_s|^2 e^{\beta s} ds ]  + E[ \int_t^T e^{\beta s}\|Z_s\|^2_{X_s} ds ] \\[4mm]
= E [ e^{\beta T} |\xi|^2 ] + 2 E [ \int_t^T e^{\beta s} V_s f(s, V_s, Z_s) ds ]+ 2 E [ \int_t^T e^{\beta s} G_s dK_s]\\[4mm]
\leq  E [e^{\beta T}|\xi|^2] + 2 E[ \int_t^T e^{\beta s} |V_s|(|f(s,0,0)|+|f(s, V_s, Z_s)-f(s,0,0)|)ds] \\[4mm]
 \quad + 2 e^{\beta T}E [ \int_t^T  G_s dK_s] \end{array}$$
 $$\begin{array}{ll}
\leq  E [e^{\beta T}|\xi|^2] + 2 E[ \int_t^T e^{\beta s}|V_s| (|f(s, 0, 0)| + c |V_s| + c \|Z_s\|_{X_s})ds] \\[4mm]
\quad  + 2 e^{\beta T}E [ \int_t^T  G_s dK_s] \\[4mm]
 \leq E [e^{\beta T}|\xi|^2] + 2E [ \int_t^Te^{\beta s}|V_s|\cdot | f(s, 0, 0)|ds]  + (2c+3c^2) E[\int_t^T e^{\beta s}|V_s|^2  ds] \\[4mm]
\quad + \dfrac{1}{3} E [ \int_t^T e^{\beta s}\|Z_s\|^2_{X_s} ds]+2 e^{\beta T}E [K_T \sup\limits_{0 \leq s\leq T} (G_s^+)  ]
\\[4mm]
\leq E [e^{\beta T}|\xi|^2] +2 E [ \sup\limits_{s\in[t,T]}e^{\frac{1}{2}\beta s}|V_s| \int_t^Te^{\frac{1}{2}\beta s} |f(s, 0, 0)| ds] \\[4mm]
\quad + (2c+3c^2) E[\int_t^T e^{\beta s}|V_s|^2  ds] + \dfrac{1}{3} E [ \int_t^T e^{\beta s}\|Z_s\|^2_{X_s} ds]\\[4mm]
\quad + \dfrac{e^{2\beta T}}{\alpha}E [ \sup\limits_{0 \leq s \leq T} (G_s^+)^2] + \alpha E [ K_T^2]\\[4mm]
\leq E [e^{\beta T}|\xi|^2] +\gamma E [ \sup\limits_{s\in[t,T]}e^{\beta s}|V_s|^2]+\dfrac{1}{\gamma}E[ (\int_t^Te^{\frac{1}{2}\beta s} |f(s, 0, 0)| ds)^2] \\[4mm]
\quad + (2c+3c^2) E[\int_t^T e^{\beta s}|V_s|^2  ds] + \dfrac{1}{3} E [ \int_t^T e^{\beta s}\|Z_s\|^2_{X_s} ds]\\[4mm]
\quad + \dfrac{e^{2\beta T}}{\alpha}E [ \sup\limits_{0 \leq s\leq T} (G_s^+)^2] + \alpha E [ K_T^2],
\end{array}
$$
where $\alpha,\gamma>0$ are two arbitrary constants. Therefore, for any $t\in[0,T],$
\begin{align}\label{apriory21}
\nonumber&
E [e^{\beta t} |V_t|^2 ]  + (\beta-2c-3c^2)E[ \int_t^T |V_s|^2 e^{\beta s} ds] + \frac{2}{3} E [ \int_t^T e^{\beta s}\|Z_s\|^2_{X_s} ds]  \\
\nonumber&\leq E [e^{\beta T}|\xi|^2]+\gamma E [ \sup_{s\in[t,T]}e^{\beta s}|V_s|^2]+\frac{1}{\gamma}E[ (\int_t^Te^{\frac{1}{2}\beta s} |f(s, 0, 0)| ds)^2]\\
&\quad  +\frac{e^{2\beta T}}{\alpha}E [ \sup_{0 \leq s \leq T} (G_s^+)^2] + \alpha E [ K_T^2].
\end{align}
We now give an estimate for $ E [ K_T^2]$. Because
$K_T  = V_0 - \xi - \int_0^T f(t, V_t, Z_t) dt $ $+ \int_0^T Z_{t}' dM_t,$
by Lemma \ref{Z2}, we deduce
\begin{align*}
E [ |K_T|^2] & \leq 4 E [  |V_0|^2 + |\xi|^2 + |\int_0^T f(t, V_t, Z_t) dt|^2 + |\int_0^T Z_{t}'dM_t|^2 ] \\
& \leq 4 E [ |V_0|^2 + |\xi|^2+ \int_0^T \|Z_t\|^2_{X_t} dt]\\
& \quad + 4  E [ (\int_0^T (|f(t,0,0)|+ c |V_t| + c \|Z_t\|_{X_t})dt)^2]
\end{align*}
\begin{align*}
& \leq 4 E [ |V_0|^2 + |\xi|^2+\int_0^T \|Z_t\|^2_{X_t} dt]+ 12  E [ (\int_0^T |f(t,0,0)|dt)^2]\\
& \quad+ 12 c^2 E[ (\int_0^T|V_t|dt)^2 ]+12 c^2E[ (\int_0^T\|Z_t\|_{X_t}dt)^2]\\
& = 4  E[|\xi|^2+ |V_0|^2] +(4+12c^2T)E[  \int_0^T \|Z_t\|^2_{X_t} dt]\\
& \quad +12 E [(\int_0^T| f(t,0,0)|dt)^2]+12c^2TE[\int_0^T|V_t|^2dt].
\end{align*}
So, there is a constant $C_1>0$ depending on $c$ and $T$ such that
\begin{align}\label{estimate_K11}
\nonumber &E[|K_T|^2] \\
\nonumber&\leq C_1( E[|\xi|^2 + |V_0|^2 +(\int_0^T |f(t,0,0)|dt)^2+ \int_0^T (|V_t|^2 + \| Z_t\|^2_{X_t}) dt])\\
\nonumber&\leq C_1(E[e^{\beta T}|\xi|^2 + |V_0|^2 +(\int_0^T e^{\frac{1}{2}\beta t}|f(t,0,0)|dt)^2]\\
&\quad + C_1E[\int_0^T e^{\beta t}(|V_t|^2 + \| Z_t\|^2_{X_t}) dt])
\end{align}
Then we consider $t=0$ in \eqref{apriory21}. Set $\alpha =\dfrac{1}{3C_1} $ and $\beta=2c+3c^2+\dfrac{2}{3}$, we obtain
\begin{align}\label{estimate_V0}
\nonumber&E[|V_0 |^2]+ E [ \int_0^T e^{\beta s}(|V_s|^2+\|Z_s\|^2_{X_s})ds]\\
 \nonumber&\leq4e^{\beta T} E [|\xi|^2]+3\gamma E [ \sup_{s\in[0,T]}e^{\beta s}|V_s|^2]+(\frac{3}{\gamma}+1)E[ (\int_0^T e^{\frac{1}{2}\beta s} |f(s, 0, 0)| ds)^2]\\
& \quad +9C_1e^{2\beta T}E [ \sup_{0 \leq s \leq T} (G_s^+)^2].
\end{align}
Then by (\ref{estimate_K11}) and (\ref{estimate_V0}), we derive
\begin{align}\label{estimate_K1}
\nonumber E[|K_T|^2] &\leq5  C_1e^{\beta T}E[|\xi|^2 ] + C_1(\frac{3}{\gamma}+2)E[(\int_0^T e^{\frac{1}{2}\beta t}|f(t,0,0)|dt)^2]\\
&\quad+3 C_1\gamma E [ \sup_{s\in[0,T]}e^{\beta s}|V_s|^2]+9C_1^2e^{2\beta T}E [ \sup_{0 \leq s \leq T} (G_s^+)^2].
\end{align}
Because $ V_t = \xi + \int_t^T f(s, V_s, Z_s ) ds + K_T -K_t -\int_t^T  Z_s' dM_s,~  0 \leq t \leq T$, we have for $t\in[0,T],$
$$\begin{array}{ll}E[\sup\limits_{0\leq t \leq T}|V_t|^2]\\[4mm]
\leq E[4|\xi|^2] + 4(\int_0^T| f(s, V_s, Z_s) |ds)^2 + 4|K_T|^2 + \sup\limits_{0\leq t \leq T}4|\int_t^T Z_s' dM_s|^2] \\[4mm]
 \leq E[4|\xi|^2 + 12( \int_0^T |f(s,0,0)|ds)^2 + 12Tc^2\int_0^T( |V_s|^2+\|Z_s\|^2_{X_s})ds] \\[4mm]
\quad+ E[4 |K_T|^2 + 4\sup\limits_{0\leq t \leq T}|\int_t^T  Z_s' dM_s|^2]\\[4mm]
\leq E[4e^{\beta T}|\xi|^2 + 12( \int_0^T e^{\frac{1}{2}\beta s}|f(s,0,0)|ds)^2 + 12Tc^2\int_0^Te^{\beta s} (|V_s|^2+\|Z_s\|^2_{X_s})ds] \\[4mm]
 \quad+E[ 4 |K_T|^2 + 4\sup\limits_{0\leq t \leq T}|\int_t^T  Z_s' dM_s|^2].
\end{array}
$$
Using Doob's inequality and Lemma \ref{Z2}, we know
\begin{align*}
 E [ \sup_{0\leq t \leq T}|\int_t^T Z_s' dM_s|^2 ] & =E [ \sup_{0\leq t \leq T}|\int_0^T Z_s' dM_s-\int_0^t Z_s' dM_s|^2]\\
  &\leq 2 E [|\int_0^T Z_s' dM_s|^2+ \sup_{0\leq t \leq T}|\int_0^t Z_s' dM_s|^2] \\
& \leq 10 E [ |\int_0^T Z'_sdM_s|^2]= 10 E[ \int_0^T \|Z_s\|^2_{X_s}ds]\\
&\leq10 E[ \int_0^T e^{\beta s}\|Z_s\|^2_{X_s}ds].
\end{align*}
Hence, with the help of (\ref{estimate_V0}) we conclude
\begin{align*}
&~ E [ \sup_{0\leq t \leq T}|V_t|^2 ]\\
& \leq4(1+5C_1)e^{\beta T}E[|\xi|^2] + (12+\frac{12 C_1}{\gamma}+8C_1)E[( \int_0^T e^{\frac{1}{2}\beta s}f(s,0,0)ds)^2]\\
&\quad+12 C_1\gamma E [ \sup_{s\in[0,T]}e^{\beta s}|V_s|^2]+36C_1^2e^{2\beta T}E [ \sup_{0 \leq s \leq T} (G_s^+)^2] \\
& \quad+ 12Tc^2E[\int_0^Te^{\beta s} (|V_s|^2+\|Z_s\|^2_{X_s})ds] + 40E[\int_0^T e^{\beta s}\|Z_s\|^2_{X_s}ds]\\
& \leq4(1+5C_1)e^{\beta T}E[|\xi|^2] + (12+\frac{12 C_1}{\gamma}+8C_1)E[( \int_0^T e^{\frac{1}{2}\beta s}f(s,0,0)ds)^2]\\
&\quad+12 C_1\gamma E [ \sup_{s\in[0,T]}e^{\beta s}|V_s|^2]+36C_1^2e^{2\beta T}E [ \sup_{0 \leq s \leq T} (G_s^+)^2] \\
& \quad+ (12Tc^2 + 40)E[\int_0^Te^{\beta s} (|V_s|^2+\|Z_s\|^2_{X_s})ds] 
\end{align*}
\begin{align*}
& \leq4(41+5C_1 +12Tc^2)e^{\beta T}E[|\xi|^2]\\
&\quad + (52+8C_1+12Tc^2+\frac{ 12C_1+36Tc^2+120}{\gamma})E[( \int_0^T e^{\frac{1}{2}\beta s}f(s,0,0)ds)^2]\\
&\quad+36C_1(C_1+3Tc^2+10)e^{2\beta T}E [ \sup_{0 \leq s \leq T} (G_s^+)^2] \\
& \quad+12(C_1+3Tc^2+10)\gamma E [ \sup_{s\in[0,T]}e^{\beta s}|V_s|^2].
\end{align*}
Noticing $E [ \sup_{s\in[0,T]}e^{\beta s}|V_s|^2]\leq e^{\beta T}E [ \sup_{s\in[0,T]}|V_s|^2]$, set $\gamma=\dfrac{1}{24(C_1+3Tc^2+10)e^{\beta T}}$, we deduce there exists a constant $C_2>0$ depending on $T$ and $c$ such that
\begin{align}\label{supV}
E [ \sup_{0\leq t \leq T}|V_t|^2 ]
 \leq C_2( E [|\xi|^2+\sup\limits_{0 \leq s \leq T} (G_s^+)^2+(\int_0^T |f(s, 0, 0)| ds)^2].
\end{align}
By (\ref{estimate_V0}), (\ref{estimate_K1}) and (\ref{supV}) we know there exists a constant $C>0$ depending on $T$ and $c$ such that
\begin{align*}
& E [ \sup_{0\leq t \leq T}|V_t|^2 ]+ E [ \int_0^T \|Z_t\|^2_{X_t} dt]+ E [|K_T|^2]\\
 & \leq C( E [|\xi|^2+\sup\limits_{0 \leq s \leq T} (G_s^+)^2+(\int_0^T |f(s, 0, 0)| ds)^2].
\end{align*}
\end{proof}
\indent Similarly we obtain the following result for the solutions to BSDEs driven by the Markov Chain.
\begin{proposition}\label{esBSDE}
Suppose $\xi\in L^2(\mathcal{F}_T)$ and $f$ satisfies \eqref{Lipch}, \eqref{Con_f}.  Let $(Y, Z)$ be the solution of the BSDE for the Markov Chain as following
$$ Y_t = \xi + \int_t^T f(s,Y_s, Z_s ) ds -\int_t^T  Z_s' dM_s,~~  0 \leq t \leq T. $$
Then there exists a constant $C'>0$ depending on the Lipschitz constants $c',c''$ of $f$ and $T$ such that
\begin{align*}
E [ \sup_{0\leq t \leq T}|Y_t|^2 ]+ E [ \int_0^T \|Z_t\|^2_{X_t} dt]
 \leq C'( E [|\xi|^2+(\int_0^T |f(s, 0, 0)| ds)^2].
\end{align*}
\end{proposition}
\section{Continuous dependence
property of solutions to RBSDEs for the Markov Chain}
\begin{proposition}\label{cdp}
Suppose $f$ satisfies \eqref{Lipch}, \eqref{c''}, \eqref{Con_f}, $\xi\in L^2(\mathcal{F}_T)$, $G^{(i)}$ satisfies \eqref{Con_g}, and $\{\varphi^{(i)}_t;~t\in[0,T]\}$ is a predictable process satisfying
\begin{equation}\label{Con_v}
 E [ \int_0^T |\varphi_t^{(i)}|^2 dt] < \infty,
\end{equation}where $i=1,2$.  Let $(V^{(i)}, Z^{(i)},K^{(i)})$ be the solution of RBSDE for the Markov Chain as following
$$ \left \{
\begin{array}{ll} V^{(i)}_t = \xi_i + \int_t^T( f(s, V^{(i)}_s, Z^{(i)}_s ) +\varphi_s^{(i)})ds + K^{(i)}_T -K^{(i)}_t -\int_t^T  (Z^{(i)}_s)' dM_s,~t \in[0,T];\\[2mm]
 V^{(i)}_t \geq G^{(i)}_t ,~~~   t \in[0,T];\\[2mm]
 \{K^{(i)}_t, t \in [0,T]\}\mbox{ is continuous and increasing, moreover, }K^{(i)}_0=0\mbox{ and} \\[1mm]
\int_0^T (V^{(i)}_s - G^{(i)}_s) dK^{(i)}_s = 0,\end{array} \right. $$
where $i=1,2$. Set
$(v,z)=(V^{(1)}-V^{(2)},Z^{(1)}-Z^{(2)}),~k=K^{(1)}-K^{(2)}.$ Then there exists a constant $\bar{C}>0$ depending on the Lipschitz constants $c',c''$ of $f$ and $T$ such that
$$
\begin{array}{ll}
\sup\limits_{t\in[0,T]}E[|v_t|^2]+E[\int_0^T \|z_t \|_{X_t}^2 dt]+\sup\limits_{t\in[0,T]}E[|k_T-k_t|^2]\\[4mm]
 \leq \bar{C}(E[|\xi_1-\xi_2|^2+
(\int_0^T  |\varphi_t^{(1)}-\varphi_t^{(2)} |dt)^2+  \sup\limits_{t\in[0,T]}|G^{(1)}_t-G^{(2)}_t|^2]).
\end{array}
$$
Moreover,
$$
\begin{array}{ll}
E[\sup\limits_{t\in[0,T]}|v_t+k_t|^2]+\int_0^T \|z_t \|_{X_t}^2 dt]\\[4mm]
 \leq \bar{C}E[|\xi_1-\xi_2|^2+
(\int_0^T  |\varphi_t^{(1)}-\varphi_t^{(2)}| dt)^2+  \sup\limits_{t\in[0,T]}|G^{(1)}_t-G^{(2)}_t|^2].
\end{array}
$$ \end{proposition}
\begin{proof} Applying the product rule to $|v_t|^2$, we have:
\begin{align}\label{v1v22}
\nonumber
&|v_t|^2\\
\nonumber
&=|\xi_1-\xi_2|^2+2 \int_t^T v_{s} ( f(s,V_s^{(1)}, Z_s^{(1)})-f(s,V_s^{(2)}, Z_s^{(2)})+(\varphi_s^{(1)}-\varphi_s^{(2)})) ds \\
& \quad +2 \int_t^T v_sdk_s -2 \int_t^T  v_{s-} z_s' dM_s - \int_t^T z_s' dL_sz_s -\int_t^T \|z_s \|_{X_s}^2 ds.
\end{align}
Since $\int_t^T (V_s^{(1)} - G^{(1)}_s) dK_s^{(1)}=\int_t^T (V_s^{(2)} - G^{(2)}_s) dK_s^{(2)}=0$, we have
\begin{align}\label{k1k2}
\nonumber
  &\int_t^T v_{s} dk_s\\
\nonumber&=\int_t^T (V_s^{(1)} - V^{(2)}_s) dK^{(1)}_s- \int_t^T (V_s^{(1)} - V^{(2)}_s) dK_s^{(2)}\\
\nonumber&=\int_t^T (V_s^{(1)} - G^{(1)}_s) dK^{(1)}_s- \int_t^T (V_s^{(2)} - G^{(2)}_s) dK_s^{(1)}\\
\nonumber
&\quad - \int_t^T (V_s^{(1)} - G^{(1)}_s) dK^{(2)}_s + \int_t^T (V_s^{(2)} - G^{(2)}_s) dK_s^{(2)}\\
\nonumber
& \quad +\int_t^T  G^{(1)}_sdK^{(1)}_s- \int_t^T  G^{(2)}_sdK^{(1)}_s - \int_t^T  G^{(1)}_sdK^{(2)}_s+ \int_t^T  G^{(2)}_sdK^{(2)}_s\\
\nonumber
& =- \int_t^T (V_s^{(2)} - G^{(2)}_s) dK_s^{(1)}-\int_t^T (V_s^{(1)} - G^{(1)}_s) dK^{(2)}_s + \int_t^T  (G^{(1)}_s-G^{(2)}_s)dk_s \\
&\leq \int_t^T  (G^{(1)}_s-G^{(2)}_s)dk_s .
\end{align}
Write $c=\max\{c',c''\}$. By (\ref{v1v22}) and (\ref{k1k2}), we obtain
$$
\begin{array}{ll}
&E[|v_t|^2+ \int_t^T \|z_s \|_{X_s}^2 ds]\\[3mm]
& \leq E[|\xi_1-\xi_2|^2]+2 E[\int_t^T  (G^{(1)}_s-G^{(2)}_s)dk_s]\\[3mm]
& \quad+2 E[\int_t^T |v_{s} |\cdot|f(s,V_s^{(1)}, Z_s^{(1)})-f(s,V_s^{(2)}, Z_s^{(2)})+(\varphi_s^{(1)}-\varphi_s^{(2)})| ds]\\[3mm]
&  \leq E[|\xi_1-\xi_2|^2]+2E[c \int_t^T |v_{s}| (|v_{s}|+\|z_s\|_{X_s})ds+\int^T_t |v_{s}|\cdot|\varphi_s^{(1)}-\varphi_s^{(2)}| ds]\\[3mm]
& \quad + 2E[\int_t^T  (G^{(1)}_s-G^{(2)}_s)dk_s]\\[3mm]
&  \leq  E[|\xi_1-\xi_2|^2]+E[\int_t^T  ((2c+3c^2+1)|v_{s}|^2+\dfrac{1}{3}\|z_s\|_{X_s}^2+|\varphi_s^{(1)}-\varphi_s^{(2)}|^2) ds]\\[3mm]
& \quad+ 2E[|k_T-k_t|\cdot \sup\limits_{s\in[t,T]}|G^{(1)}_s-G^{(2)}_s|]
\end{array}
$$
\begin{equation}\label{esvtzt}
\begin{array}{ll}
&\leq  E[|\xi_1-\xi_2|^2]+E[\int_t^T  ((2c+3c^2+1)|v_{s}|^2+\dfrac{1}{3}\|z_s\|_{X_s}^2) ds]\\[3mm]
& \quad+
E[\int_t^T  |\varphi_s^{(1)}-\varphi_s^{(2)}|^2 ds]+  \epsilon E[ \sup\limits_{s\in[t,T]}|G^{(1)}_s-G^{(2)}_s|^2]+\dfrac{1}{\epsilon} E[|k_T-k_t|^2],
\end{array}\end{equation}
where $\epsilon>0$ is an arbitrary positive constant. Since for any $t\in[0,T],$
\begin{align*}v_t&=\xi_1-\xi_2+\int^T_t(( f(s,V_s^{(1)}, Z_s^{(1)})-f(s,V_s^{(2)}, Z_s^{(2)}))+(\varphi_s^{(1)}-\varphi_s^{(2)})) ds\\
& \quad+k_T-k_t-\int^T_tz_s'dM_s,~~t\in[0,T],\end{align*}
 by Doob's inequality and Lemma \ref{Z2} we deduce
\begin{align}\label{eskn}
\nonumber
&E[|k_T-k_t|^2]\\[2mm]
\nonumber
&\leq4E[ |\xi_1-\xi_2|^2+|v_t|^2+|\int^T_tz_s'dM_s|^2]\\
\nonumber
&\quad+4E[|\int^T_t(( f(s,V_s^{(1)}, Z_s^{(1)})-f(s,V_s^{(2)}, Z_s^{(2)}))+(\varphi_s^{(1)}-\varphi_s^{(2)})) ds|^2]\\
\nonumber
& \leq4E[ |\xi_1-\xi_2|^2+|v_t|^2+\int^T_t\|z_s\|_{X_s}^2 ds]\\
\nonumber
&\quad+8E[(\int^T_t| f(s,V_s^{(1)}, Z_s^{(1)})-f(s,V_s^{(2)}, Z_s^{(2)})|ds)^2+(\int^T_t|\varphi_s^{(1)}-\varphi_s^{(2)}| ds)^2]\\
\nonumber
&\leq4E[ |\xi_1-\xi_2|^2+|v_t|^2+\int^T_t\|z_s\|_{X_s}^2 ds]\\
\nonumber
&\quad+8 c^2 E[(\int^T_t(|v_s|+\|z_s\|_{X_s}) ds)^2]+8E[(\int^T_t|\varphi_s^{(1)}-\varphi_s^{(2)}| ds)^2]\\
\nonumber
&\leq4E[ |\xi_1-\xi_2|^2+|v_t|^2+\int^T_t\|z_s\|_{X_s}^2 ds]\\
\nonumber
&\quad+16 c^2 TE[\int^T_t(|v_s|^2+\|z_s\|_{X_s}^2) ds]+8E[(\int^T_t|\varphi_s^{(1)}-\varphi_s^{(2)}| ds)^2]\\
\nonumber
&\leq4E[ |\xi_1-\xi_2|^2]+4E[|v_t|^2]+8E[(\int^T_t|\varphi_s^{(1)}-\varphi_s^{(2)}| ds)^2]\\
&\quad+16 c^2 TE[\int^T_t|v_s|^2 ds]+4(1+4 c^2T)E[\int^T_t\|z_s\|_{X_s}^2ds].
\end{align}
Set $\epsilon=8(1+4 c^2T)$ in inequality (\ref{esvtzt}). With the help of inequality (\ref{eskn}), from inequality (\ref{esvtzt}) we derive for any $t\in[0,T],$
\begin{align}\label{esvtzt1}
\nonumber
&\dfrac{1}{2}E[|v_t|^2+\dfrac{1}{6} \int_t^T \|z_s \|_{X_s}^2 ds]\\
\nonumber
&\leq \dfrac{3}{2} E[|\xi_1-\xi_2|^2]+E[\int_t^T  (2c+3c^2+2)|v_{s}|^2 ds]\\
&\quad+
2E[(\int_t^T  |\varphi_s^{(1)}-\varphi_s^{(2)}| ds)^2]+  8(1+4 c^2T)E[ \sup\limits_{s\in[t,T]}|G^{(1)}_s-G^{(2)}_s|^2].
\end{align}
Using Gronwall's inequality, by (\ref{esvtzt1}) we have for any $t\in[0,T],$
$$\begin{array}{ll}
E[|v_t|^2]\\[3mm]
 \leq (3E[|\xi_1-\xi_2|^2+
4(\int_0^T  |\varphi_s^{(1)}-\varphi_s^{(2)}| ds)^2+ 16(1+4 c^2T) \sup\limits_{s\in[0,T]}|G^{(1)}_s-G^{(2)}_s|^2])\\[3mm]
~~~\cdot e^{2(2c+3c^2+2)(T-t)}.
\end{array}$$
So there exists a constant $C_3>0$ depending on the constants $c$ and $T$ such that
\begin{equation}\label{esvtzt2}
\begin{array}{ll}
\sup\limits_{t\in[0,T]}E[|v_t|^2]\\[3mm]
 \leq C_3E[|\xi_1-\xi_2|^2+
(\int_0^T  |\varphi_t^{(1)}-\varphi_t^{(2)}| dt)^2+  \sup\limits_{t\in[0,T]}|G^{(1)}_t-G^{(2)}_t|^2].
\end{array}
\end{equation}
Noting $E[\int^T_t|v_s|^2 ds]\leq(T-t)\sup\limits_{t\in[0,T]}E[|v_t|^2]$ for $t\in[0,T]$, by (\ref{eskn}), (\ref{esvtzt1}) and (\ref{esvtzt2}) we know
there exists a constant $C_4>0$ depending on the constants $c$ and $T$ such that
\begin{equation}\label{essupvt}
\begin{array}{ll}
\sup\limits_{t\in[0,T]}E[|v_t|^2]+E[\int_0^T \|z_t \|_{X_t}^2 dt]+\sup\limits_{t\in[0,T]}E[|k_T-k_t|^2]\\[4mm]
 \leq C_4E[|\xi_1-\xi_2|^2+
(\int_0^T  |\varphi_t^{(1)}-\varphi_t^{(2)}| dt)^2 +  \sup\limits_{t\in[0,T]}|G^{(1)}_t-G^{(2)}_t|^2].
\end{array}
\end{equation}
Because
\begin{align*}v_t+k_t&=\xi_1-\xi_2+\int^T_t(( f(s,V_s^{(1)}, Z_s^{(1)})-f(s,V_s^{(2)}, Z_s^{(2)}))+(\varphi_s^{(1)}-\varphi_s^{(2)})) ds\\
& \quad+k_T-\int^T_tz_s'dM_s,~~t\in[0,T],\end{align*}
we deduce
\begin{align*}&E[\sup\limits_{t\in[0,T]}|v_t+k_t|^2]\\
&\leq 4E[|\xi_1-\xi_2|^2]+4E[|k_T|^2]+4E[\sup\limits_{t\in[0,T]}|\int^T_tz_s'dM_s|^2]\\
& \quad+4E[(\int^T_0|( f(s,V_s^{(1)}, Z_s^{(1)})-f(s,V_s^{(2)}, Z_s^{(2)}))+(\varphi_s^{(1)}-\varphi_s^{(2)})| ds)^2].\end{align*}
By (\ref{eskn}), set $t=0$, we derive
$$
\begin{array}{ll}E[|k_T|^2]
 \leq&4E[ |\xi_1-\xi_2|^2]+4\sup\limits_{s\in[0,T]}E[|v_s|^2]+8E[(\int^T_0|\varphi_s^{(1)}-\varphi_s^{(2)}| ds)^2]\\[4mm]
&+16 c^2 TE[\int^T_0|v_s|^2 ds]+4(1+4 c^2T)E[\int^T_0\|z_s\|_{X_s}^2ds],
\end{array}
$$
Then by (\ref{essupvt}), we obtain there exists a constant $C_5>0$ depending on the constants $c$ and $T$ such that
$$
\begin{array}{ll}
E[\sup\limits_{t\in[0,T]}|v_t+k_t|^2]+\int_0^T \|z_t\|_{X_t}^2 dt]\\[4mm]
 \leq C_5E[|\xi_1-\xi_2|^2+
(\int_0^T  |\varphi_t^{(1)}-\varphi_t^{(2)}| dt)^2+  \sup\limits_{t\in[0,T]}|G^{(1)}_t-G^{(2)}_t|^2].
\end{array}
$$
\end{proof}
\indent Similarly we obtain the following result for the solutions to BSDEs driven by the Markov Chain.
\begin{proposition}\label{cdpbsde}
Suppose $f$ satisfies \eqref{Lipch},\eqref{Con_f}, $\xi_i\in L^2(\mathcal{F}_T)$, and $\{\varphi^{(i)}_t;~t\in[0,T]\}$ is a predictable process satisfying (\ref{Con_v}).  Let $(Y^{(i)}, Z^{(i)})$ be the solution of the BSDE for the Markov Chain as following
$$ Y^{(i)}_t = \xi_i + \int_t^T( f(s, Y^{(i)}_s, Z^{(i)}_s ) +\varphi_s^{(i)})ds -\int_t^T  (Z^{(i)}_s)' dM_s,~~t \in[0,T], $$
where $i=1,2$. Set
$(y,z)=(Y^{(1)}-Y^{(2)},Z^{(1)}-Z^{(2)}).$ Then there exists a constant $\bar{C}'>0$ depending on the Lipschitz constant $c$ of $f$ and $T$ such that
$$
\begin{array}{ll}
E[\sup\limits_{t\in[0,T]}|y_t|^2]+\int_0^T \|z_t\|^2_{X_t}dt]
 \leq \bar{C}'E[|\xi_1-\xi_2|^2+
(\int_0^T  |\varphi_t^{(1)}-\varphi_t^{(2)}| dt)^2].
\end{array}
$$ \end{proposition}
\section{A Comparison theorem for one-dimensional RBSDEs driven by the Markov chain}
\indent Suppose $(Y,Z,K)$ and $(V, U, J)$ are the solutions of the following two RBSDEs for the Markov Chain, respectively,
$$
\left \{
\begin{array}{ll}
Y_t = \xi_1+ \int_t^T f(s, Y_s, Z_s ) ds + K_T -K_t -\int_t^T  Z'_sdM_s,~~  0 \leq t \leq T;\\[2mm]
Y_t \geq G_t ,~~~  0 \leq t \leq T; \\[2mm]
 \{K_t, t \in [0,T]\}\mbox{ is continuous and increasing, moreover, }K_0=0\mbox{ and} \\[1mm]
\int_0^T (Y_s - G_s)  dK_s = 0,\end{array} \right.
$$
and
$$
\left \{
\begin{array}{ll}
V_t = \xi_2+ \int_t^T g(s, V_s, U_s ) ds + J_T -J_t -\int_t^T  U'_sdM_s,~~  0 \leq t \leq T;\\[2mm]
V_t \geq G_t ,~~~  0 \leq t \leq T;\\[2mm]
 \{J_t, t \in [0,T]\}\mbox{ is continuous and increasing, moreover, }J_0=0\mbox{ and} \\[1mm]
\int_0^T (V_s - G_s) dJ_s = 0,\end{array} \right. $$
\begin{theorem} \label{CT}
Assume $\xi_1,\xi_2\in L^2(\mathcal{F}_T)$, $G$ satisfies \eqref{Con_g}, and  $f,g$ satisfy \eqref{Lipch}, \eqref{c''} and \eqref{Con_f}. If $\xi_1 \leq \xi_2 $, a.s., and $f(t,v, z) \leq g(t,v, z)$, a.e., a.s., for any $t\in[0,T],(v,z)\in \mathbb{R}\times \mathbb{R}^N$;
 then
$$P( Y_t\leq V_t,\text{ for any } t \in [0,T])=1$$
and
$$P( K_t\geq J_t,\text{ for any } t \in [0,T])=1.$$
\end{theorem}
\begin{proof}
For each $n \in \mathbb{N}$, $(t,v,z)\in [0,T]\times \mathbb{R}\times \mathbb{R}^N$, define:
$$f_n (t,v,z)= f(t,v,z)+n(v-G_t)^{-}.$$
 It is clear that $f_n$ is Lipschitz continuous. For each $n \in \mathbb{N}$, consider BSDE
\[Y_t^{n} = \xi_1 + \int_t^T f_n(s, Y_s^{n}, Z_s^{n}) ds - \int_t^T  (Z_{s}^{n})' dM_s,~~~0\leq t\leq T.\]
By Lemma \ref{existence}, for each $n \in \mathbb{N}$, the above equation has a unique solution $(Y^n, Z^n)\in L^2_{\mathcal{F}}(0,T;\mathbb{R})\times P^2_{\mathcal{F}}(0,T;\mathbb{R}^N)$. For each $n \in \mathbb{N}$,
define:
$$K_t^{n}= n \int_0^t (Y^{n}_s -G_s)^- ds.$$
On the other hand, for each $n \in \mathbb{N}$, $(t,v,z)\in [0,T]\times \mathbb{R}\times \mathbb{R}^N$, define:
$$g_n (t,v,z)= g(t,v,z)+n(v-G_t)^{-}.$$
 It is clear that $g_n$ is Lipschitz continuous. For each $n \in \mathbb{N}$, consider BSDE
\[V_t^{n} = \xi_2 + \int_t^T g_n(s, V_s^{n}, U_s^{n}) ds - \int_t^T  (U_{s}^{n})' dM_s,~~~0\leq t\leq T.\]
By Lemma \ref{existence}, for each $n \in \mathbb{N}$, the above equation has a unique solution $(V^n, U^n)\in L^2_{\mathcal{F}}(0,T;\mathbb{R})\times P^2_{\mathcal{F}}(0,T;\mathbb{R}^N)$. For each $n \in \mathbb{N}$,
define:
$$J_t^{n}= n \int_0^t (V^{n}_s -G_s)^- ds.$$
For each $n \in \mathbb{N}$, $f_n$ satisfies Assumption \ref{ass0} and $f_n \leq g_n$. Therefore from Lemma \ref{CTBSDE}, for each $n \in \mathbb{N}$,
$$ P(Y^{n}_t \leq V^{n}_t, ~ \text{for any } t \in [0,T])=1.$$
 That is, for any $n \in \mathbb{N}$, there exists a subset $B_n \subseteq \Omega$ such that $P(B_n) =1$ and for any $\omega \in B_n$, $Y_t^n(\omega) \leq V_t^{n}(\omega)$, for any $t \in [0,T]$. Let $\tilde{B} = \bigcap\limits_{n=1}^{\infty} B_n$, hence $P(\tilde{B}) =1$ and for each $\omega \in \tilde{B}$, $Y_t^n(\omega) \leq V_t^{n}(\omega)$, for any $t \in [0,T]$, $n \in \mathbb{N}$. Hence
 $$P( \sup_{n \in \mathbb{N}}Y_t^n \leq  \sup_{n \in \mathbb{N}}V_t^{n},~\text{ for any } t \in [0,T])=1.$$
From the proof of existence in  Ramarimbahoaka, Yang and Elliott \cite{RYE} we have $Y_t = \sup\limits_{n \in \mathbb{N}} Y_t^n$ and $V_t= \sup\limits_{n \in \mathbb{N}} V_t^n,$ $t \in [0,T]$. Thus,
 $$P( Y_t \leq V_t,~\text{ for any } t \in [0,T])=1.$$
Also from the proof of existence in  Ramarimbahoaka, Yang and Elliott \cite{RYE} we have when $n \rightarrow \infty$,
\begin{equation*}
E[\sup_{0\leq t\leq T}|K_t^n -K_t|^2 ] \rightarrow 0 ~~~\text{ and }~~~E[\sup_{0\leq t\leq T}|J_t^n -J_t|^2 ] \rightarrow 0.
\end{equation*}
Then for any $\epsilon>0,$
\[\lim_{n\rightarrow\infty} P(\sup_{0\leq t\leq T} |K^{n}_t-K_t| > \epsilon)=0~~\text{ and }~~\lim_{n\rightarrow\infty} P(\sup_{0\leq t\leq T} |J^{n}_t-J_t| > \epsilon)=0 .\]
So there exists a subsequence $\{n_k;~ k \in \mathbb{N}\} \subset \{n;~ n \in \mathbb{N}\}$ and a subsequence $\{n_{k_m};~m \in \mathbb{N}\} \subset \{n_k;~ k \in \mathbb{N}\}$
such that
$$\lim_{k\rightarrow\infty}  \sup_{0\leq t \leq T} |K_t^{n_k}-K_t|=0, ~ \text{ a.e.}$$
and
$$\lim_{m\rightarrow\infty}  \sup_{0\leq t \leq T} |J_t^{n_{k_m}}-J_t|=0, ~ \text{ a.e.}$$
Therefore,
\begin{align*}&P(\lim_{m\rightarrow\infty}  K_t^{n_{k_m}}=K_t,~\text{ for any } t \in [0,T])=1;\\[2mm]
\text{ and }&~P(\lim_{m\rightarrow\infty}  J_t^{n_{k_m}}=J_t,~\text{ for any } t \in [0,T])=1.\end{align*}
That is, there is a subset $A\subseteq \Omega$ such that $P(A)=1$ and for any $\omega\in A$, $\lim\limits_{m\rightarrow\infty}  K_t^{n_{k_m}}(\omega)=K_t(\omega)$ and
$\lim\limits_{m\rightarrow\infty}  J_t^{n_{k_m}}(\omega)=J_t(\omega)$, for any $t\in[0,T]$.\\
Now we prove for any $\omega\in\tilde{B}$, $K_t^{n_{k_m}}(\omega)\geq J_t^{n_{k_m}}(\omega)$, for any $t\in[0,T],m \in \mathbb{N}$.
Noticing on $\tilde{B}$, $Y_t^{n_{k_m}}\leq V_t^{n_{k_m}}$, for any $t\in[0,T],m \in \mathbb{N}$, there are three cases:
\begin{enumerate}
\item If $(t,\omega)\in[0,T]\times\tilde{B}$ such that $G_t(\omega)\leq Y^{n_{k_m}}_t(\omega)\leq V^{n_{k_m}}_t(\omega)$, for any $m \in \mathbb{N}$, then
$$ (Y^{n_{k_m}}_t -G_t)^- (\omega)=(V^{n_{k_m}}_t -G_t)^- (\omega)=0,~ \text{ for any }m \in \mathbb{N}.$$
By the definition of $K^{n_{k_m}}$ and $J^{n_{k_m}}$ we deduce $K^{n_{k_m}}_t(\omega)=J^{n_{k_m}}_t(\omega)$, for any $m \in \mathbb{N}$.
\item If $(t,\omega)\in[0,T]\times\tilde{B}$ such that $Y^{n_{k_m}}_t(\omega)\leq G_t(\omega)\leq V^{n_{k_m}}_t(\omega)$, for any $m \in \mathbb{N}$, then
$$ (Y^{n_{k_m}}_t -G_t)^- (\omega)\geq0=(V^{n_{k_m}}_t -G_t)^- (\omega),~ \text{ for any }m \in \mathbb{N}.$$
Then we derive $K^{n_{k_m}}_t(\omega)\geq J^{n_{k_m}}_t(\omega)$, for any $m \in \mathbb{N}$.
\item If $(t,\omega)\in[0,T]\times\tilde{B}$ such that $Y^{n_{k_m}}_t(\omega)\leq V^{n_{k_m}}_t(\omega)\leq G_t(\omega)$, for any $m \in \mathbb{N}$, then
$$ (Y^{n_{k_m}}_t -G_t)^- (\omega)= G_t(\omega) -Y^{n_{k_m}}_t(\omega)\geq G_t(\omega) -V^{n_{k_m}}_t(\omega)=(V^{n_{k_m}}_t -G_t)^- (\omega),$$
for any $m \in \mathbb{N}$. Then we obtain $K^{n_{k_m}}_t(\omega)\geq J^{n_{k_m}}_t(\omega)$, for any $m \in \mathbb{N}$.
\end{enumerate}
Hence we know for any $\omega\in\tilde{B}$, $K_t^{n_{k_m}}(\omega)\geq J_t^{n_{k_m}}(\omega)$, for any $t\in[0,T],m \in \mathbb{N}$. Thus, for any $\omega\in\tilde{B}$,
$$\lim _{m\rightarrow\infty}K_t^{n_{k_m}}(\omega)\geq \lim _{m\rightarrow\infty}J_t^{n_{k_m}}(\omega),~~\text{ for any }t\in[0,T].$$
Therefore, for any $\omega\in A\cap\tilde{B}$,
$$K_t(\omega)=\lim _{m\rightarrow\infty}K_t^{n_{k_m}}(\omega)\geq \lim _{m\rightarrow\infty}J_t^{n_{k_m}}(\omega)=J_t(\omega),\text{ for any } t \in [0,T].$$
Because $P(A\cap\tilde{B})=1$, we conclude
$$P( K_t\geq J_t,\text{ for any } t \in [0,T])=1.$$
\end{proof}
\section{Conclusion}
\indent In this paper, we have
provided an estimate for the solutions of RBSDEs, derived a continuous dependence property for their solutions with
respect to the parameters of the equations, and established a comparison result
for the solutions of RBSDEs driven by a Markov chain.

\end{document}